\documentclass[graybox]{svmult}

\def\Dbar{\leavevmode\lower.6ex\hbox to 0pt{\hskip-.23ex \accent"16\hss}D}

\def\bZ{{\mbox{\bf Z}}}
\def\bC{{\mbox{\bf C}}}

\def\paf{{\mbox{\rm PAF}}}
\def\psd{{\mbox{\rm PSD}}}
\def\dft{{\mbox{\rm DFT}}}

\makeindex             

\begin{document}

\title*{D-optimal matrices of orders 118, 138, 150, 154 and 174}
\author{Dragomir {\v{Z}}. {\Dbar}okovi{\'c} and Ilias S. Kotsireas}
\institute{Dragomir {\v{Z}}. {\Dbar}okovi{\'c}
\at University of Waterloo, Department of Pure Mathematics, Waterloo, Ontario, N2L 3G1, Canada
\email{djokovic@math.uwaterloo.ca}
\and
Ilias S. Kotsireas
\at  Wilfrid Laurier University, Department of Physics
\& Computer Science, Waterloo, Ontario, N2L 3C5, Canada
\email{ikotsire@wlu.ca}}
%
%

\maketitle
\vspace{-1cm}
\centerline{ {\bf Dedicated to Hadi Kharaghani on his 70th birthday} } $ $ \\

\abstract{
We construct supplementary difference sets (SDS) with parameters
$(59;28,22;21)$, $(69;31,27;24)$, $(75;36,29;28)$,
$(77;34,31;27)$ and $(87;38,36;31)$. These SDSs give D-optimal designs (DO-designs) of two-circulant type of orders 118,138,150,154 and 174. Until now, no DO-designs of orders 138,154 and 174 were known. While a DO-design (not of two-circulant type) of order 150 was constructed previously by Holzmann and Kharaghani, no such design of two-circulant type was known. The smallest undecided order for DO-designs is now 198. We use a novel property of the compression map to speed up some computations.
}

\section{Introduction}

Let $v$ be any positive integer. We say that a sequence $A=[a_0,a_1,\ldots,a_{v-1}]$ is a {\em binary sequence} if  $a_i\in\{1,-1\}$ for all $i$. We denote by $\bZ_v=\{0,1,\ldots, v-1\}$ the ring of integers modulo $v$.
There is a bijection from the set of all binary sequences of
length $v$ to the set of all subsets of $\bZ_v$ which assigns
to the sequence $A$ the subset $\{i\in\bZ_v:a_i=-1\}$.
If $X\subseteq\bZ_v$, then the corresponding binary sequence $[x_0,x_1,\ldots,x_{v-1}]$ has $x_i=-1$ if $i\in X$ and $x_i=+1$ otherwise. We associate to $X$ the cyclic matrix $C_X$ of order $v$ having this sequence as its first row.

Assume temporarily that $v$ is odd. If $A$ is a $\{+1,-1\}$-matrix of size $2v\times2v$, then it is well known \cite{Ehlich}, \cite{Wojtas} that
\begin{equation} \label{Bound}
\det A \le 2^v (2v-1)(v-1)^{v-1}.
\end{equation}
Moreover, this inequality is strict if $2v-1$ is not a sum of two squares \cite[Satz 5.4]{Ehlich}.
In this paper we are interested only in the case when equality holds in (\ref{Bound}). In that case we say that $A$ is a {\em D-optimal design} (DO-design) of order $2v$. Hence, in the context of DO-designs we shall assume that $2v-1$ is a sum of two squares. For the problem of maximizing the determinant of square $\{+1,-1\}$-matrices of any fixed order, we refer the interested reader to \cite{Brent} and its references.

Many DO-designs of order $2v$ can be constructed by using supplementary difference sets with suitable parameters $(v;r,s;\lambda)$. We recall that these parameters are nonnegative integers such that $\lambda(v-1)=r(r-1)+s(s-1)$.
(See section \ref{SDS} below for the formal definition of SDSs over a finite cyclic group.) For convenience, we also introduce the parameter $n=r+s-\lambda$. Without any loss of generality we may assume that the parameter set is {\em normalized} which means that we have $v/2\ge r\ge s\ge0$.
The SDSs that we need are those for which $v=2n+1$. We refer to them as {\em D-optimal SDS}.

The feasible parameter sets for the D-optimal SDS can be easily
generated by using the following proposition.
Apparently this fact has not been observed so far.

\begin{proposition} \label{Formule}
Let $P$ be the set of ordered pairs $(x,y)$ of integers $x,y$
such that $x\ge y\ge0$. Let $Q$ be the set of normalized feasible
parameter sets $(v;r,s;\lambda)$ for D-optimal SDSs. Thus, it is required that $v=2n+1$ where $n=r+s-\lambda$. Then the map
$P\to Q$, given by the formulas
\begin{eqnarray} \label{Par-formule}
v &=& 1+x(x+1)+y(y+1), \\
r &=& {x+1\choose2}+{y\choose2}, \\
s &=& {x\choose2}+{y+1\choose2}, \\
\lambda &=& {x\choose2}+{y\choose2}, \end{eqnarray}
is a bijection.
\end{proposition}

We leave the proof to the interested reader. Note that
$n={x+1\choose2}+{y+1\choose2}$.

Let $(X,Y)$ be a D-optimal SDS with parameters $(v;r,s;\lambda)$. Then the associated matrices $C_X$ and $C_Y$ satisfy the equation
\begin{equation} \label{MatEqDO}
C_X C_X^T + C_Y C_Y^T = (2v-2)I_v + 2J_v,
\end{equation}
where $I_v$ is the identity matrix, $J_v$ is the matrix with all entries equal to $1$, and the superscript T denotes the transposition of matrices. One can verify that the matrix
\begin{equation} \label{DO-Mat}
\left( \begin{array}{rr} C_X & C_Y \\ -C_Y^T & C_X^T
\end{array} \right)
\end{equation}
is a DO-design of order $2v$. We say that the DO-designs obtained by this construction (due to Ehlich and Wojtas) are of {\em two-circulant type} (2c type).

In the range $0<v<100$, $v$ odd integer, the condition that $2v-1$ is a sum of two squares rules out the following 14 odd integers:
$11$, $17$, $29$, $35$, $39$, $47$, $53$, $65$, $67$, $71$, $81$, $83$, $89$, $85$. In the remaining 36 cases the DO-designs of order $2v$ are known (see \cite{KO:2007,DK:JCD:2012} and their references) except for $v=69,77,87,99$. In the case $v=75$ the only known DO-design of order $150$ \cite{HK} is not of 2c type.

Our main result is the construction of DO-designs of 2c type
for orders $2v$ with $v=59,69,75,77,87$. This is accomplished by constructing the SDSs with parameters $(59;28,22;21)$, $(69;31,27;24)$, $(75;36,29;28)$, $(77;34,31;27)$ and $(87;38,36;31)$, respectively. These SDSs give DO-designs of 2c type of orders 118,138,150,154 and 174. Until now, no DO-designs of orders 138,154 and 174 were known, and no DO-design of 2c type and order 150 was known. However, a DO-design (not of 2c type) of order 150 was constructed previously by Holzmann and Kharaghani \cite{HK}. The first DO-design of order 118 was constructed in
\cite{FKS:2004}, we provide two more non-equivalent examples.
The main tool that we use in our constructions is the method of compression of SDSs which we developed in our recent paper \cite{DK:JCD:2014}. This method uses a nontrivial factorization $v=md$ and so it can be applied only when $v$ is a composite integer. In the cases mentioned above we used the factorizations with $m=3$ or $m=7$.

In section \ref{SDS} we recall the definition of SDSs over finite cyclic groups, and in section \ref{Compression} we establish a relationship between power density functions of a complex
sequence of length $v=md$ and its compressed sequence of length $d$. This relationship was used to speed up some of the computations.

In section \ref{sec:DO-results} we list the 2, 19, 3, 1 and 3 non\-equivalent SDSs for the DO-designs of order 118, 138, 150, 154 and 174, respectively. Consequently, for orders less than 200 only the DO-design of order 198 remains unknown. In some cases, for a given odd integer $v$ such that $2v-1$ is a sum of two squares, there exist more than one feasible parameter set $(v;r,s;\lambda)$ with $v=2n+1$ and $v/2\ge r\ge s$ (see \cite[Table I]{DK:JCD:2012}). For instance, this is the case for
$v=85$. In that case there are two feasible parameter sets and an SDS is known only for one of them.

In the appendix we list D-optimal SDSs, one per the parameter set $(v;r,s;\lambda)$, for all $v<100$ with two exceptions where such
SDS is not known.

Finally, we point out two misprints in our recent paper \cite{DK:JCD:2014}. (i) The first formula in \cite[eq. (16)]{DK:JCD:2014} should read $\beta_0=v(tv-4n)+4n$. (ii) In item 4
of Remark 1 the formula should read $\sum (v-2k_i)^2=4v$.

\section{Supplementary difference sets} \label{SDS}

We recall the definition of SDSs. Let $k_1,\ldots,k_t$ be positive integers and $\lambda$ an integer such that
\begin{equation} \label{par-lambda}
\lambda (v - 1) = \sum_{i=1}^t k_i (k_i - 1).
\end{equation}
and let $X_1,\ldots,X_t$ be subsets such that
\begin{equation} \label{kard-ki}
\end{equation}

\begin{definition}
We say that the subsets $X_1,\ldots,X_t$ of $\bZ_v$ with
$|X_i|=k_i$ for $i\in\{1,\ldots,t\}$ are {\em supplementary difference sets (SDS)} with parameters $(v;k_1,\ldots,k_t;\lambda)$, if for every nonzero element $c\in\bZ_v$ there are exactly $\lambda$ ordered triples $(a,b,i)$ such that $\{a,b\}\subseteq X_i$ and $a-b=c \pmod{v}$.
\end{definition}

These SDS are defined over the cyclic group of order $v$, namely the additive group of the ring $\bZ_v$. More generally SDS can be defined over any finite abelian group, and there are also further generalizations where the group may be any finite group. However, in this paper we shall consider only the cyclic case.

In the context of an SDS, say $X_1,\ldots,X_t$, with parameters $(v;k_1,\ldots,k_t;\lambda)$, we refer to the subsets $X_i$ as the {\em base blocks} and we introduce an additional parameter, $n$, defined by:
\begin{equation} \label{par-n}
n = k_1 + \cdots + k_t - \lambda.
\end{equation}

If $x$ is an indeterminate, then the quotient ring
$\bC[x]/(x^v-1)$ is isomorphic to the ring of complex circulant matrices of order $v$. Under this isomorphism $x$ corresponds to the cyclic matrix with first row $[0,1,0,0,\ldots,0]$. By applying this isomorphism to the identity \cite[(13)]{DK:JCD:2014}, we obtain that the following matrix identity holds
\begin{equation} \label{MatNorm}
\sum_{i=1}^t C_i C_i^T = 4nI_v +(tv-4n)J_v,
\end{equation}
where $C_i=C_{X_i}$ is the cyclic matrix associated to $X_i$.

In this paper we are mainly interested in SDSs $(X,Y)$ with two base blocks, i.e., $t=2$. Then if $v=2n+1$ the identity (\ref{MatNorm}) reduces to the identity (\ref{MatEqDO}).

\section{Compression of SDSs} \label{Compression}

Let $A$ be a complex sequence of length $v$. For the standard definitions of periodic autocorrelation functions $(\paf_A)$, discrete Fourier transform $(\dft_A)$, power spectral density $(\psd_A)$ of $A$, and the definition of complex complementary sequences, we refer the reader to our paper \cite{DK:JCD:2014}. If we have a collection of complex complementary sequences of length $v=dm$, then we can compress them to obtain complementary sequences of length $d$. We refer to the ratio $v/d=m$ as the {\em compression factor}. Here is the precise definition.

\begin{definition}
Let $A = [a_0,a_1,\ldots,a_{v-1}]$ be a complex sequence of length $v = dm$ and set
\begin{equation} \label{koef-kompr}
a_j^{(d)}=a_j+a_{j+d}+\ldots+a_{j+(m-1)d}, \quad
j=0,\ldots,d-1.
\end{equation}
Then we say that the sequence
$A^{(d)} = [a_0^{(d)},a_1^{(d)},\ldots,a_{d-1}^{(d)}]$
is the {\em $m$-compression} of $A$.
\end{definition}

Let $X,Y$ be a D-optimal SDS with parameters $(v;r,s;\lambda)$
and let $n=r+s-\lambda$. Thus $v=2n+1$. Assume that $v=md$ is a nontrivial factorization. Let $A,B$ be their associated binary sequences. Then the $m$-compressed sequences $A^{(d)},B^{(d)}$ form a complementary pair. In general they are not binary sequences, their terms belong to the set
$\{m,m-2,\ldots,-m+2,-m\}$.
The search for such pairs $X,Y$ is broken into two stages: first we construct the candidate complementary sequences $A^{(d)},B^{(d)}$ of length $d$, and second we lift each of them and search to find the D-optimal pairs $(X,Y)$. Each of the stages requires
a lot of computational resources. There are additional theoretical results that can be used to speed up these computations. Some of them are descirbed in \cite{DK:JCD:2014},
namely we use ``bracelets'' and ``charm bracelets'' to speed up
the first stage. We give below a new theoretical result, which we used to speed up the second stage.

\begin{theorem}
Let $A=[a_0,a_1,\ldots,a_{v-1}]$ be a complex sequence of length $v=md$ where $m,d>1$ are integers. Let $A^{(d)}=[a^{(d)}_0,a^{(d)}_1,\ldots,a^{(d)}_{d-1}]$ be the $m$-compression of the sequence $A$. Then
\begin{equation} \label{ms}
\psd_A(ms)=\psd_{A^{(d)}}(s), \quad s=0,1,\ldots,d-1.
\end{equation}
\end{theorem}
\begin{proof}
For the discrete Fourier transform of $A$ we have
\begin{equation}
\dft_A(ms) = \sum_{j=0}^{v-1} a_j \omega^{mjs}
= \sum_{j=0}^{v-1} a_j \omega_0^{js}
= \sum_{j=0}^{d-1} a^{(d)}_j \omega_0^{js},
\end{equation}
where $\omega=\exp(2\pi i/v)$ and $\omega_0=\omega^m=\exp(2\pi i/d)$. Hence, by using the Wiener-Khinchin theorem (i.e., that
$\psd=\dft\circ\paf$), we have
\begin{eqnarray*}
\psd_A(ms) &=& \left| \dft_A(ms) \right|^2 \\
&=& \sum_{j,k=0}^{d-1}  a^{(d)}_j \omega_0^{js}
\overline{ a^{(d)}_k } \omega_0^{-ks} \\
&=& \sum_{j,k=0}^{d-1}  a^{(d)}_j \overline{ a^{(d)}_k}
\omega_0^{(j-k)s} \\
&=& \sum_{r=0}^{d-1} \left( \sum_{k=0}^{d-1} a^{(d)}_{k+r}
\overline{ a^{(d)}_k } \right) \omega_0^{rs} \\
&=& \sum_{r=0}^{d-1} \paf_{A^{(d)}}(r) \omega_0^{rs} \\
&=& \dft(\paf_{A^{(d)}})(s) \\
&=& \psd_{A^{(d)}}(s).
\end{eqnarray*}
\end{proof}

\section{Computational results for DO-designs}
\label{sec:DO-results}

All solutions are in the canonical form defined in \cite{Djokovic:AnnComb:2011} and since they are different, this implies that they are pairwise non\-equivalent.
Taking into account our new results, the open cases for DO-designs with $v < 200$ are:
\begin{eqnarray*}
&& 99, 111, 115, 117, 123, 129, 135, 139, 141, 147, 153, 159,\\
&&  163, 167, 169, 175, 177, 185, 187, 189, 195, 199.
\end{eqnarray*}

\subsection{D-optimal SDS with parameters $( 59; 28, 22; 21 )$}

An SDS with these parameters has been constructed in
\cite{FKS:2004}, it is equivalent to the one listed in the
appendix. We have constructed two more such SDS, not equivalent
to the one mentioned above.

\begin{eqnarray*}
1) &&
\{0,1,2,4,5,6,8,9,10,12,14,19,21,24,25,28,30,31,33,37,41,42,\\
&& 43,45,46,52,53,54\},\{0,1,2,3,5,6,7,8,13,15,16,18,21,23,27,\\
&& 31,32,35,38,41,48,52\},\\
2) &&
\{0,1,2,3,5,7,8,11,13,14,15,17,18,19,23,25,26,31,32,33,35,38,\\
&& 40,42,47,51,53,56\},\{0,1,3,4,5,6,8,13,14,15,17,23,25,26,\\
&& 29,30,33,36,40,41,45,46\}.
\end{eqnarray*}

\subsection{D-optimal SDS with parameters $( 69; 31, 27; 24 )$}

Until now, 69 was the smallest positive odd integer $v$ for which the existence of DO-designs of order $2v$ was undecided. We have constructed 19 nonequivalent SDSs for the above parameter set, which give 19 DO-designs of order 138.

\begin{eqnarray*}
1) &&
\{0,1,3,4,6,9,10,11,13,14,17,18,20,22,26,28,29,32,33,34,39,\\
&& 41,43,45,46,48,51,59,60,62,63\},\{0,2,3,4,8,9,10,11,12,15,16,\\
&& 17,21,25,26,32,33,35,36,37,39,41,46,51,54,57,59\},\\
2) &&
\{0,1,3,4,5,6,7,8,9,10,12,16,18,19,23,25,27,28,31,32,33,39,\\
&& 40,41,42,47,52,53,58,60,63\},\{0,1,3,4,5,8,10,12,13,14,18,20,\\
&& 23,26,27,30,31,38,41,43,44,47,51,53,55,58,63\},\\
3) &&
\{0,2,3,5,6,7,8,11,12,14,15,16,21,22,23,24,26,30,31,32,35,\\
&& 36,37,41,43,46,49,53,54,56,58\},\{0,1,2,3,4,8,9,11,14,16,17,\\
&& 20,24,28,31,33,35,37,38,41,42,45,47,53,57,59,60\},\\
4) &&
\{0,1,2,4,5,6,7,8,10,12,15,17,20,23,24,25,28,30,34,36,37,\\
&& 40,42,46,47,48,49,51,55,62,63\},\{0,1,2,3,4,5,7,11,12,14,16,\\
&& 18,19,21,22,28,31,32,37,38,43,47,51,52,55,60,63\},\\
5) &&
\{0,1,2,3,4,7,8,10,11,13,16,18,21,24,25,26,27,30,32,33,34,\\
&& 37,39,41,44,45,54,55,58,59,60\},\{0,1,2,5,6,8,10,11,12,14,15,\\
&& 17,23,24,30,32,34,36,39,40,43,44,51,56,59,61,63\},\\
6) &&
\{0,1,3,4,5,6,7,10,14,15,16,18,22,24,25,26,27,28,32,33,34,\\
&& 39,41,44,48,52,53,55,57,60,61\},\{0,2,3,5,6,8,11,12,13,15,18,\\
&& 19,23,25,26,28,33,37,39,40,42,43,44,48,52,58,64\},\\
7) &&
\{0,1,2,3,4,5,7,10,11,13,14,19,20,22,23,25,28,30,31,33,35,\\
&& 36,37,41,43,45,50,54,57,58,64\},\{0,1,2,3,4,6,7,11,12,15,20,\\
&& 22,25,26,27,30,31,32,38,39,42,44,46,48,55,59,62\},\\
8) &&
\{0,2,3,5,6,7,9,10,11,12,14,15,18,20,22,24,25,29,34,35,36,\\
&& 37,38,45,46,49,51,53,55,59,66\},\{0,1,2,3,5,6,11,13,14,17,20,\\
&& 21,22,27,28,29,33,34,38,41,43,46,50,52,53,56,64\},\\
9) &&
\{0,1,2,3,4,5,6,8,9,10,13,14,15,16,20,22,25,27,28,32,35,36,\\
&& 37,43,45,46,49,52,54,56,61\},\{0,1,3,5,7,8,11,13,14,15,19,23,\\
&& 26,28,29,30,33,39,43,44,45,49,52,57,60,61,63\},\\
10) &&
\{0,1,2,3,4,6,7,8,13,14,17,19,21,22,25,26,28,29,30,34,37,40,\\
&& 41,42,44,45,50,51,54,59,64\},\{0,1,3,4,5,6,7,10,13,15,16,17,\\
&& 22,24,26,31,33,34,37,39,40,45,47,55,57,59,65\},\\
11) &&
\{0,1,2,4,5,6,7,8,11,14,15,18,19,20,21,23,25,28,29,30,37,39,\\
&& 41,42,43,45,47,50,54,57,62\},\{0,1,2,3,5,8,9,11,13,16,17,21,\\
&& 24,26,27,30,33,36,40,41,42,47,51,52,53,62,64\},\\
12) &&
\{0,1,2,3,4,6,8,9,10,11,13,15,16,17,20,23,24,28,29,31,34,38,\\
&& 39,40,43,49,51,53,55,56,59\},\{0,1,2,4,6,9,10,12,13,14,18,19,\\
&& 23,26,30,33,35,36,44,45,47,50,51,52,58,60,63\},\\
13) &&
\{0,1,2,3,4,6,7,9,11,13,16,17,20,25,26,27,28,30,32,35,36,37,\\
&& 40,43,46,47,52,57,58,60,64\},\{0,1,3,4,5,9,10,12,13,17,18,19,\\
&& 20,21,25,27,31,32,34,38,46,48,50,51,54,56,59\},\\
14) &&
\{0,1,2,3,4,5,6,8,10,13,14,15,16,18,20,23,24,26,29,32,33,36,\\
&& 39,41,43,48,50,53,54,55,61\},\{0,1,2,3,6,7,8,9,13,15,17,22,23,\\
&& 26,30,33,34,37,42,45,46,48,50,51,59,60,65\},\\
15) &&
\{0,1,2,3,5,6,7,8,11,15,18,20,21,23,25,29,30,31,32,38,39,41,\\
&& 42,43,44,49,51,55,57,60,65\},\{0,1,2,4,5,8,9,11,12,13,17,20,\\
&& 23,24,26,28,30,33,34,37,39,40,47,48,53,55,65\},\\
16) &&
\{0,1,2,4,5,8,10,11,14,15,16,18,19,22,23,25,28,29,30,34,37,38,\\
&& 40,42,45,47,50,52,53,54,63\},\{0,1,2,3,5,6,9,14,16,17,18,19,22,\\
&& 26,28,30,32,37,38,39,44,45,47,49,50,56,65\},\\
17) &&
\{0,1,2,3,4,6,7,8,12,14,15,17,22,23,24,26,27,28,30,33,37,40,\\
&& 41,45,48,51,54,56,57,58,64\},\{0,1,2,5,6,7,11,13,14,15,17,21,\\
&& 23,26,30,31,33,35,37,38,40,42,43,48,51,52,60\},\\
18) &&
\{0,1,2,3,6,7,8,9,12,13,15,20,21,23,24,26,28,30,31,32,33,35,\\
&& 38,42,43,44,48,52,56,59,62\},\{0,1,2,4,5,6,8,9,11,17,18,19,21,\\
&& 23,28,32,33,37,40,43,44,46,48,53,54,56,62\},\\
19) &&
\{0,1,2,3,7,8,9,10,11,12,13,16,18,19,22,24,25,27,29,32,35,36,\\
&& 39,41,45,46,50,52,54,56,57\},\{0,1,2,3,4,6,8,12,15,16,20,21,22,\\
&& 24,29,30,34,35,38,41,42,45,48,50,53,58,60\}.\\
\end{eqnarray*}

\subsection{D-optimal SDS with parameters $( 75; 36, 29; 28 )$}

Until now, no DO-design of 2c type and order 150 was known. We have constructed 3 nonequivalent SDSs for the above parameter set. They give 3 DO-designs of 2c type and order 150.
\begin{eqnarray*}
1) && \{0,1,2,3,4,5,8,9,10,12,13,16,17,19,22,25,27,28,30,32,33,34,38,\\
&& 40,42,44,47,49,51,54,57,60,61,65,66,67\}, \\
&& \{0,1,2,4,5,6,7,9,10,12,16,17,21,24,25,30,31,32,35,38,39,41,\\
&& 43,45,51,52,61,63,64\}, \\
2) &&
\{0,1,2,3,7,8,9,10,11,13,15,16,20,21,23,24,26,27,30,32,34,36,\\
&& 38,39,44,45,48,49,50,52,55,60,64,65,66,69\}, \\
&& \{0,2,3,4,5,6,7,9,12,13,14,17,21,22,24,26,31,34,37,39,40,\\
&& 46,50,53,54,55,57,61,69\}, \\
3) &&
\{0,1,4,5,6,7,9,12,14,15,16,17,20,22,24,25,26,28,30,32,33,39,\\
&& 40,41,44,45,46,47,49,53,56,59,62,65,67,69\},\\
&& \{0,2,4,5,6,7,9,11,12,16,18,19,22,23,29,30,31,33,34,40,43,\\
&& 44,48,49,52,53,58,60,61\}. \\
\end{eqnarray*}

\subsection{D-optimal SDS with parameters $( 77; 34, 31 ; 27 )$}

We have constructed only one solution.
\begin{eqnarray*}
1) &&\{0,2,3,4,5,6,9,10,12,14,17,19,22,23,24,26,29,30,32,33,36,\\
&& 37,39,44,45,48,50,54,58,60,61,63,69,71\}, \\
&& \{0,1,2,4,5,6,9,10,12,14,17,20,21,22,23,24,28,29,35,38,40, \\
&& 44,45,49,51,52,53,54,60,64,65\}.\\
\end{eqnarray*}

\subsection{D-optimal SDS with parameters $( 87; 38, 36 ; 31 )$}

We have constructed 3 nonequivalent solutions.
\begin{eqnarray*}
1) &&\{0,1,2,3,4,5,6,8,10,12,16,18,22,23,24,25,32,33,36,37,38,\\
&& 39,43,46,47,50,54,56,57,61,62,63,66,69,71,74,80,83\}, \\
&& \{0,1,2,5,6,8,10,11,13,17,18,19,21,23,24,26,27,29,33,36,38, \\
&& 40,43,45,48,49,51,52,53,54,58,65,66,69,77,78\}, \\
2) && \{0,1,2,4,5,7,10,11,14,15,17,19,22,23,24,25,27,29,30,35,\\
&& 36,39,42,44,50,51,54,55,57,59,61,65,66,68,73,77,78,81\}, \\
&& \{0,1,2,3,4,5,6,7,8,13,14,19,21,22,27,28,30,31,32,36,38,39, \\
&& 40,43,45,47,48,49,54,57,59,61,67,70,73,77\}, \\
3) &&
\{0,1,3,5,6,8,9,11,12,15,16,18,19,20,25,27,28,29,31,33,40, \\
&& 41,45,46,47,50,51,55,58,61,62,64,68,69,70,72,76,78\}, \\
&& \{0,1,2,3,4,7,8,10,12,14,15,17,19,24,27,28,29,33,34,35,36, \\
&& 37,40,42,47,48,50,51,53,58,63,66,67,69,78,82\}. \\
\end{eqnarray*}

\section{Acknowledgements}
The authors wish to acknowledge generous support by NSERC.
This work was made possible by the facilities of the Shared Hierarchical Academic Research Computing Network (SHARCNET) and Compute/Calcul Canada.
We thank a referee for his suggestions.

\section{Appendix: D-optimal SDS with $v<100$}

We list here all D-optimal parameter sets $(v;r,s;\lambda)$ with
$v/2\ge r\ge s$ and $v<100$ and for each of them (with two exceptions) we give one DO-design of 2c type by recording the two base blocks of the corresponding SDS. In the two exceptional cases we indicate by a question mark that such designs are not yet known. In particular, this means that DO-designs of order
$2v<200$, with $v$ odd, exist for all feasible orders (those
for which $2v-1$ is a sum of two squares) except for $v=99$. This list will be useful to interested readers as examples of such designs are spread out over many papers in the literature.
For the benefit of the readers interested in binary sequences we mention that these SDS give two binary sequences of length $v$ with PAF +2, i.e., D-optimal matrices.

$$
\begin{array}{ll}
(v;r,s;\lambda) & \mbox{{\rm Base blocks}} \\
\hline \\
(3;1,0;0) & \{0\}, \quad \emptyset \\
(5;1,1;0) & \{0\}, \quad  \{0\} \\
(7;3,1;1) & \{0,1,3\}, \quad  \{0\} \\
(9;3,2;1)  & \{0,1,4\},  \quad  \{0,2\}  \\
(13;4,4;2) & \{0,1,4,6\},   \quad   \{0,1,4,6\}  \\
(13;6,3;3) & \{0,1,2,4,7,9\},  \quad   \{0,1,4\}  \\
(15;6,4;3) & \{0,1,2,4,6,9 \},  \quad   \{0,1,4,9 \}  \\
(19;7,6;4) & \{0,1,2,3,7,11,14 \},\quad\{0,2,5,6,9,11 \}  \\
(21;10,6;6) & \{0,1,2,3,4,6,8,11,12,16 \},
\quad \{0,1,3,7,10,15 \}  \\
(23;10,7;6) & \{0,1,3,4,5,7,8,12,14,18\},
\quad\{0,1,2,7,9,12,15\}\\
(25;9,9;6)& \{0,1,2,4,7,11,14,15,20\},
\quad\{ 0,1,2,4,6,9,10,12,17 \}  \\
(27;11,9;7) & \{ 0,1,3,4,5,9,10,11,13,16,19 \}, \quad
\{ 0,1,2,4,8,12,15,17,22 \} \\
(31;15,10;10) & \{0,1,2,3,5,6,7,11,13,15,16,18,23,24,27\}, \\
& \{ 0,2,3,5,6,8,12,19,20,27 \}  \\
(33;13,12;9) & \{ 0,1,2,4,5,6,8,10,15,17,20,25,26 \}, \\
&   \{ 0,2,3,5,6,9,12,13,17,19,24,25 \}  \\
(33;15,11;10) & \{0,1,2,3,4,5,8,10,12,13,14,18,19,22,26\}, \\
& \{ 0,1,2,5,8,11,15,17,20,22,28 \}  \\
(37;16,13;11))& \{0,1,2,3,4,7,8,11,13,15,16,18,23,24,27,33\},\\
& \{ 0,1,2,4,8,10,13,14,18,20,21,23,32 \}  \\
(41;16,16;12)& \{0,1,2,3,5,7,8,9,13,18,19,22,23,26,32,34\}, \\
& \{ 0,1,3,4,6,8,11,13,15,16,17,23,24,27,30,36 \} \\
(43;18,16;13)& \{0,1,2,3,4,7,9,11,12,13,16,19,22,24,25,29,30,36 \}, \\
&  \{0,1,2,4,5,6,9,14,16,17,20,24,26,31,33,39 \} \\
(43;21,15;15) &  \{0,1,2,3,4,5,6,7,11,12,13,14,17,20,24,25,28,30,31,34,39 \},\\
& \{0,2,3,4,7,9,12,14,16,22,24,30,31,34,39 \}  \\
(45;21,16;15)& \{ 0,1,2,3,5,6,8,10,12,13,14,20,21,22,25,28, 29,32,34,35,42 \}, \\
& \{ 0,1,2,4,5,6,10,11,14,16,19,22,29,31,33,40 \} \\
(49;22,18;16) &  \{0,1,2,3,4,5,6,9,11,13,14,19,20,21,23,26,27,30,35,38,40,42\},\\
& \{0,1,3,4,5,8,9,13,15,19,21,24,26,27,30,37,43,44 \}  \\
(51;21,20;16) &
\{0,2,4,5,6,9,11,12,13,18,19,21,22,26,27,28,30,33,38,39,41\},\\
& \{ 0,1,2,4,5,6,9,10,12,14,17,22,24,25,28,31,35,37,41,42 \} \\
(55;24,21;18) & \{ 0,1,2,3,6,8,10,11,13,14,17,19,20,21,24,26, 28,29,33,34,40,\\
& 41,43,44\},\  \{0,1,2,3,6,7,9,11,12,15,19,21,25,29,34,36,37, \\
& 38,40,45,50 \} \\
(57;28,21;21) &
\{ 0,1,2,3,4,5,8,9,10,11,13,16,17,19,21,22,23,24,27,31,34, \\
& 36,37,38,41,43,49,50 \},\  \{ 0,1,3,4,7,9,11,13,15,16,20,25,\\
& 26,29,30,35,37,40,41,43,48 \} \\
(59;28,22;21) &
\{0,2,3,5,6,8,9,10,13,15,16,17,19,23,25,26,27,29,30,34,38,\\
& 39,41,43,44,45,53,56 \},\  \{ 0,1,2,3,5,7,8,10,12,13,19,\\
& 20,22,24,28,32,33,37,38,44,45,51 \} \\
\end{array}
$$

$$
\begin{array}{ll}
(v;r,s;\lambda) &  \mbox{{\rm Base blocks}} \\
\hline \\
(61;25,25;20) & \{ 0,2,4,7,8,9,10,12,13,18,20,23,24,25,26, 29,32,33,34,38,41, \\
& 44,48,51,52\},\  \{0,1,2,4,6,7,8,12,13,14,15,16,19,23,29,30, \\
& 32,34,36,39,41,44,49,50,53 \} \\
(63;27,25;21) & \{
0,1,2,3,5,7,10,11,12,15,18,21,23,24,25,26,31,32,36,37,40, \\
& 43,44,47,49,51,53 \},\  \{ 0,2,4,6,7,8,9,10,11,12,16,20,21,24,27, \\
& 30,33,38,39,40,45,47,55,56,60 \} \\
(63;29,24;22) & \{
0,1,2,3,4,6,7,11,12,13,14,20,21,22,25,26,27,30,33,35,36, \\
& 38,39,42,46,48,50,53,57\},\  \{0,1,3,5,7,8,10,11,13,14,16,18,\\
& 19,23,30,33,34,35,39,40,48,52,54,56 \} \\
(69;31,27;24) &
\{0,1,3,4,6,9,10,11,13,14,17,18,20,22,26,28,29,32,33,34,39, \\
& 41,43,45,46,48,51,59,60,62,63\},\  \{0,2,3,4,8,9,10,11,12,15,\\
& 16,17,21,25,26,32,33,35,36,37,39,41,46,51,54,57,59\} \\
(73;31,30;25) &
\{ 0,1,2,3,4,5,7,9,11,12,16,17,21,22,25,26,30,32,34,37,38,43, \\
& 44,45,46,49,52,54,56,59,62\},\  \{0,1,3,4,7,8,9,11,15,16,17,18,\\
& 21,23,26,27,28,29,31,33,40,42,46,47,50,53,56,62,63,65 \} \\
(73;36,28;28) &
\{0,1,3,4,6,7,9,10,12,13,14,15,19,20,21,25,27,28,29,30,31, \\
&36,38,39,41,42,43,46,50,51,54,55,57,59,61,63\},\\
& \{0,1,4,6,7,11,13,14,18,20,21,22,23,24,26,30,31,35,38,40,48, \\
& 51,53,54,58,59,63,65 \} \\
(75;36,29;28) &
\{ 0,1,2,3,4,5,8,9,10,12,13,16,17,19,22,25,27,28,30,32,33,34, \\
& 38,40,42,44,47,49,51,54,57,60,61,65,66,67\}, \\
& \{0,1,2,4,5,6,7,9,10,12,16,17,21,24,25,30,31,32,35,38,39, \\
& 41,43,45,51,52,61,63,64 \} \\
(77;34,31;27) &
\{ 0,2,3,4,5,6,9,10,12,14,17,19,22,23,24,26,29,30,32,33,36, \\
& 37,39,44,45,48,50,54,58,60,61,63,69,71\}, \\
& \{0,1,2,4,5,6,9,10,12,14,17,20,21,22,23,24,28,29,35,38,40, \\
& 44,45,49,51,52,53,54,60,64,65 \} \\
(79;37,31;29) &
\{0,1,2,3,4,5,6,9,12,13,14,16,18,23,24,30,31,32,33,35,38,39, \\
& 40,44,46,48,52,53,56,57,58,61,64,67,69,72,73 \}, \\
& \{ 0,1,3,4,6,8,10,11,13,14,15,17,21,22,27,28,30,32,33,34, \\
& 37,44,46,47,50,52,53,55,65,69,75\} \\
(85;36,36;30) &
\{ 0,1,2,3,5,6,8,9,12,13,15,22,24,26,28,29,33,34,35,36,38,40, \\
& 41,46,48,49,51,52,56,57,60,66,70,75,78,80\},\  \\
& \{0,2,3,4,5,6,8,11,12,17,18,19,20,21,22,25,29,31,33,36,37, \\
& 38,42,43,46,47,55,57,58,61,64,66,68,73,74,81 \} \\
(85;39,34;31) & ? \\
(87;38,36;31) &
\{ 0,1,2,3,4,5,6,8,10,12,16,18,22,23,24,25,32,33,36,37,38,39, \\
& 43,46,47,50,54,56,57,61,62,63,66,69,71,74,80,83\}, \\
& \{ 0,1,2,5,6,8,10,11,13,17,18,19,21,23,24,26,27,29,33,36,38, \\
& 40,43,45,48,49,51,52,53,54,58,65,66,69,77,78 \} \\
\end{array}
$$

$$
\begin{array}{ll}
(v;r,s;\lambda) &  \mbox{{\rm Base blocks}} \\
\hline \\
(91;45,36;36) &
\{ 0,2,4,5,6,8,9,10,11,12,13,14,17,18,19,21,24,25,27,30,33,34, \\
& 35,36,37,38,44,45,47,48,51,52,56,57,59,64,66,67,69,71,74,75, \\
& 80,84,85\},\ \{0,2,4,6,9,10,11,14,15,16,20,22,24,27,29,31,32, \\
& 34,37,38,46,49,50,51,52,53,60,63,64,66,69,70,72,76,77,85\}\\
(93;42,38;34) &
\{0,1,4,5,6,7,8,10,15,16,17,19,22,23,26,29,30,32,33,34,35,38, \\
& 40,41,45,46,47,49,53,55,60,63,65,66,70,72,73,74,77,80,82,84 \},\  \\
& \{ 0,1,2,3,4,6,8,10,11,12,13,15,16,22,24,26,27,30,31,32,35, \\
& 40,44,47,48,49,52,53,54,60,62,64,67,70,73,82,83,88\} \\
(93;45,37;36) &
\{ 0,2,3,4,6,7,8,9,11,13,14,16,18,19,20,22,23,24,26,31,34,35, \\
& 37,38,39,41,43,44,47,52,53,55,59,62,63,64,66,69,70,74,75,76,\\
& 81,83,86\},\ \{0,1,2,3,6,7,10,11,12,15,16,18,19,20,26,28,29,30,\\
& 33,36,40,42,51,52,53,55,57,58,60,65,66,74,77,79,80,85,87\}\\
(97;46,39;37) &
\{ 0,1,2,4,6,7,8,9,11,12,14,15,17,21,22,24,25,26,28,29,34,35, \\
& 36,38,44,45,47,49,51,52,53,55,57,63,64,67,68,69,73,76,78,81,\\
& 82,83,86,94\},\ \{0,1,2,3,6,8,11,12,16,17,18,20,23,25,27,28,29, \\
& 30,36,37,38,41,44,45,49,51,57,60,61,62,63,64,67,69,73,76,83, \\
& 91,94 \} \\
(99;43,42;36) & ? \\
\end{array}
$$

\end{document}